\documentclass[11pt, reqno]{amsart}
\usepackage{mathtools}

\usepackage[usenames,dvipsnames]{color}
\usepackage[OT2, T1]{fontenc}
\usepackage{url}
\usepackage{amsmath}
\usepackage{array}
\usepackage{graphicx}
\usepackage{amssymb}
\usepackage{amsthm}
\usepackage[colorlinks=true,linkcolor=blue,citecolor=orange]{hyperref}
\usepackage{hyperref}
\usepackage{paralist}
\usepackage{colonequals}	
\usepackage{color}
\usepackage{mathabx}			
\usepackage{amscd}
\usepackage[all,cmtip]{xy}			
\usepackage{sseq}				
\usepackage{verbatim}
\usepackage{parskip}
\usepackage{microtype}
\usepackage[in]{fullpage}
\usepackage{mathrsfs} 	
\usepackage{cleveref}
\usepackage{enumitem}
\usepackage{moreenum}	
\usepackage[alphabetic,lite]{amsrefs} 
\usepackage{stmaryrd}	
\usepackage[foot]{amsaddr}
\usepackage{subfiles}
\usepackage{ragged2e}
\usepackage{stackengine}

\newcommand{\brems}{\begin{rems} \hfill \begin{enumerate}[label=\b{\thenumberingbase.},ref=\thenumberingbase]}

\newcommand{\erems}{\end{enumerate} \end{rems}}
\newcommand{\begs}{\begin{egs} \hfill \begin{enumerate}[label=\b{\thenumberingbase.},ref=\thenumberingbase]}

\newcommand{\eegs}{\end{enumerate} \end{egs}}

\newcommand{\bsm}{\begin{smallmatrix}}
\newcommand{\esm}{\end{smallmatrix}}
\newcommand{\blem}{\begin{lemma}}
\newcommand{\elem}{\end{lemma}}

\newcommand{\bconj}{\begin{conj}}
\newcommand{\econj}{\end{conj}}
\newcommand{\bprob}{\begin{Problem}}
\newcommand{\eprob}{\end{Problem}}

\newcommand{\bq}{\begin{Q}}
\newcommand{\eq}{\end{Q}}
\newcommand{\benum}{\begin{enumerate}[label={{\upshape(\alph*)}}]}
\newcommand{\benuma}{\begin{enumerate}[label={{\upshape(\arabic*)}}]}
\newcommand{\benumb}{\begin{enumerate}[label={{\upshape\b{\arabic*.}}}]}
\newcommand{\benumr}{\begin{enumerate}[label={{\upshape(\roman*)}}]}
\newcommand{\eenum}{\end{enumerate}}
\newcommand{\bitem}{\begin{itemize}}
\newcommand{\eitem}{\end{itemize}}
\newcommand{\bc}{}
\newcommand{\bd}{\begin{defn}}
\newcommand{\ed}{\end{defn}}

\newcommand{\beg}{\begin{eg}}
\newcommand{\eeg}{\end{eg}}

\newcommand{\bcl}{\begin{claim}}
\newcommand{\ecl}{\end{claim}}

\newcommand{\q}{\quad}

\newcommand{\ba}{\begin{aligned}}
\newcommand{\ea}{\end{aligned}}
\newcommand{\be}{\begin{equation}}
\newcommand{\ee}{\end{equation}}
\newcommand{\bpf}{\begin{proof}}
\newcommand{\epf}{\end{proof}}
\newcommand{\bthm}{\begin{thm}}
\newcommand{\ethm}{\end{thm}}
\newcommand{\bprop}{\begin{prop}}
\newcommand{\eprop}{\end{prop}}
\newcommand{\bcor}{\begin{cor}}
\newcommand{\ecor}{\end{cor}}
\newcommand{\brem}{\begin{rem}}
\newcommand{\erem}{\end{rem}}

\usepackage{aliascnt}
\newaliascnt{numberingbase}{subsection}
\numberwithin{equation}{numberingbase}

\newtheoremstyle{thms}{0.5em}{0.5em}{\itshape}{}{\bfseries}{.}{ }{}
\theoremstyle{thms}
\newtheorem{conj}[numberingbase]{Conjecture}

\newtheorem{cor}[numberingbase]{Corollary}

\newtheorem{lemma}[numberingbase]{Lemma}

\newtheorem{lem}[numberingbase]{Lemma}

\newtheorem{prop}[numberingbase]{Proposition}

\newtheorem{Q}[numberingbase]{Question}
\newtheorem{thm}[numberingbase]{Theorem}

\newtheoremstyle{claims}{0.5em}{0.5em}{}{}{\itshape}{.}{ }{}
\theoremstyle{claims}
\newtheorem{claim}[equation]{Claim}

\newtheoremstyle{defs}{0.5em}{0.5em}{}{}{\bfseries}{.}{ }{}
\theoremstyle{defs}
\newtheorem{defn}[numberingbase]{Definition}

\newtheorem{eg}[numberingbase]{Example}
\newtheorem*{egs}{Examples}
\newtheorem{rem}[numberingbase]{Remark}

\newtheorem*{rems}{Remarks}

\Crefname{claim}{Claim}{Claims}
\Crefname{bclaim}{Claim}{Claims}
\Crefname{sublemma}{Lemma}{Lemmas}
\Crefname{conj}{Conjecture}{Conjectures}
\Crefname{cor}{Corollary}{Corollaries}
\Crefname{defn}{Definition}{Definitions}
\Crefname{eg}{Example}{Examples}
\Crefname{prop}{Proposition}{Propositions} 
\Crefname{Q}{Question}{Questions}
\Crefname{rem}{Remark}{Remarks}
\Crefname{thm}{Theorem}{Theorems}
\Crefname{Theorem}{Theorem}{Theorems}
\Crefname{variant}{Variant}{Variants}
\Crefname{caution}{Caution}{Cautions}

\theoremstyle{thms}
\newtheorem{thm-tweak}[subsection]{Theorem}
\Crefname{thm-tweak}{Theorem}{Theorems}
\newtheorem{lemma-tweak}[subsection]{Lemma}
\Crefname{lemma-tweak}{Lemma}{Lemmas}
\newtheorem{cor-tweak}[subsection]{Corollary}
\Crefname{cor-tweak}{Corollary}{Corollaries}
\newtheorem{prop-tweak}[subsection]{Proposition}
\Crefname{prop-tweak}{Proposition}{Propositions} 
\newtheorem{conj-tweak}[subsection]{Conjecture}
\Crefname{conj-tweak}{Conjecture}{Conjectures} 
\newtheorem{q-tweak}[subsection]{Question}
\Crefname{q-tweak}{Question}{Questions} 

\theoremstyle{defs}
\newtheorem{defn-tweak}[subsection]{Definition}
\Crefname{defn-tweak}{Definition}{Definitions}
\newtheorem{eg-tweak}[subsection]{Example}
\Crefname{eg-tweak}{Example}{Examples}
\newtheorem*{rems-tweak}{Remarks}
\newtheorem{rem-tweak}[subsection]{Remark}
\Crefname{rem-tweak}{Remark}{Remarks}

\newtheoremstyle{subsection-tweak}
   {2pt}
   {3pt}%
   {}
   {}%
   {\bfseries}
   {}%
   {.5em}
   {\thmnumber{\@{#1}{}\@{#2}.}%
    \thmnote{~{\bfseries#3.}}}    
    
\theoremstyle{subsection-tweak}
\newtheorem{pp}[numberingbase]{}
\newcommand{\bpp}{\begin{pp}}
\newcommand{\epp}{\end{pp}}

\theoremstyle{subsection-tweak}
\newtheorem{pp-tweak}[subsection]{}

\makeatletter
\def\@tocline#1#2#3#4#5#6#7{
    \begingroup 
    \@ifempty{#4}{}{}

    \parindent\z@ \leftskip#3\relax \advance\leftskip\@tempdima\relax
    #5\hskip-\@tempdima
      \ifcase #1
       \or\or \hskip 2em \or \hskip 1em \else \hskip 3em \fi%
      #6\nobreak\relax
    \dotfill\hbox to\@pnumwidth{\@tocpagenum{#7}}\par
    \nobreak
    \endgroup
 }
 \def\l@section{\@tocline{1}{0pt}{1pc}{}{}}

\renewcommand{\tocsection}[3]{%
  \indentlabel{\@ifnotempty{#2}{\makebox[1.3em][l]{%
    \ignorespaces#1 \bfseries{#2}.\hfill}}}\bfseries{#3}
    \vspace{-5pt}}

\renewcommand{\tocsubsection}[3]{%
  \indentlabel{\@ifnotempty{#2}{\hspace*{-0.5em}\makebox[2.1em][l]{%
    \ignorespaces#1#2.\hfill}}}#3
    \vspace{-5pt}}

\makeatother 
\makeatletter
\@namedef{subjclassname@2020}{%
  \textup{2020} Mathematics Subject Classification}
\makeatother

\author{Nguyen Mac Nam Trung}
\address{\scriptsize Institute of Mathematics, Vietnam Academy of Science and Technology, 18 Hoang Quoc Viet, 10072 Hanoi, Vietnam}
\email{\scriptsize nmntrung@math.ac.vn}

\date{\today}

\begin{document}


\title{Serre conjecture II for pseudo-reductive groups}

\maketitle

\begin{abstract} The Serre conjecture II predicts that every torsor under a semisimple, simply connected, algebraic group over a field of cohomological dimension at most $2$ and of degree of imperfection at most $1$ has a rational point. We generalize this conjecture to pseudo-reductive groups and prove their equivalence. In particular, we show that every torsor under a pseudo-semisimple, simply connected group over a global function field or a non-archimedean local field always has a rational point.
 \end{abstract}

\vspace{-45pt}

\hypersetup{
    linktoc=page,  
}
\renewcommand*\contentsname{}
\q\\
\tableofcontents

\section{Serre conjecture II} 
In 1962, Serre \cite[Conjecture~II, page~65]{Ser62} predicted that every torsor under a semisimple, simply connected, algebraic group over a perfect field of cohomological dimension at most $2$ should be trivial. Serre \cite[Remarques, page~146]{Ser94} also formulated the conjecture over an imperfect field. He anticipated that over a field of cohomological dimension at most $2$ and of degree of imperfection at most $1$, every torsor under a semisimple, simply connected, algebraic group should be trivial. This statement is also known as Serre~conjecture~II. The conjecture was settled for global function fields by Harder ~\cite[Satz A]{Har75} and non-archimedean local fields by Bruhat\textendash Tits~\cite[Th\'eor\`em~4.7(ii)]{BT87}. See the survey \cite{Gil10} for a detailed review of the state of the art on the Serre~conjecture~II. The goal of this article is to extend the results of Harder and Bruhat\textendash Tits to pseudo-reducitve groups. In fact, we prove a more general result: we formulate a Serre conjecture~II for pseudo-reductive groups and prove this is equivalent to the original Serre~conjecture~II.

The arithmetic analogue of semisimple algebraic groups are pseudo-semisimple groups. Recall that a \emph{pseudo-reductive} $k$-group is a connected, affine, smooth $k$-group that has
no nontrivial unipotent normal $k$-subgroups and a \emph{pseudo-semisimple} group is a perfect pseudo-reductive group, where $k$ is a field. The analogies between pseudo-semisimple groups and semisimple algebraic groups is a foundational topic in the theory of pseudo-reductive groups, developed by Tits and Conrad-Gabber-Prasad \cite{CGP15,CP16}. We find that the notion of simply connectedness for connected, affine, smooth algebraic groups proposed in \cite[2.5.2]{BCS25} is a perfect analogy for the notion of simply connectedness for semisimple groups, in particular, to formulate a Serre conjecture II for pseudo-reductive groups: for a field $k,$ an affine, connected, smooth $k$-group $G$ is \emph{simply connected} if the $\overline{k}$-group $G_{\overline{k}}/\mathscr{R}_{\mathrm{u},\,\overline{k}}(G)$ is semisimple and simply connected, where $\mathscr{R}_{\mathrm{u},\,\overline{k}}(G_{\overline{k}})$ denotes the geometric unipotent radical of $G$. 
\bconj[Serre conjecture II for pseudo-reductive groups] Let $k$ be a field of cohomological dimension at most 2 and of degree of imperfection at most 1. For a pseudo-semisimple, simply connected $k$-group $G$, then $$\mathrm{H}^1(k,G)=\{\ast\}.$$ 
\econj
\bthm[Corollary \ref{main-thm}]\label{main-thm-intro} Let $k$ be a field of cohomological dimension at most $2$ and of degree of imperfection at most $1.$ The following are equivalent
\benumr 
\item  $\mathrm{H}^1(k,G)=\{*\}$ for every semisimple, simply connected $k$-group $G;$
    \item $\mathrm{H}^1(k,G)=\{*\}$ for every pseudo-semisimple, simply connected $k$-group $G.$
\eenum
\ethm

By the known cases of the classical Serre conjecture II, we deduce a new vanishing theorem.
\begin{cor} Let $k$ be a field and $G$ a pseudo-semisimple, simply connected $k$-group.
\benum
    \item If $k$ is a non-archimedean local field, then $\mathrm{H}^1(k,G)=\{\ast\}.$
    \item If $k$ is a global function field, then $\mathrm{H}^1(k,G)=\{\ast\}.$
\eenum
\end{cor}
\bpp[An overview of the proof and of the paper] The first step of the proof is to reduce Theorem \ref{main-thm-intro} to absolutely pseudo-simple $k$-groups. We utilize the map from the product of minimal pseudo-simple normal $k$-subgroups of $G$ and the fact that every pseudo-semisimple, simply connected $k$-group $G$ is always of minimal type to reduce to the case when $G$ is pseudo-simple. By a standard Galois descent argument, we even reduce to the case when $G$ is absolutely pseudo-simple.  

For an absolutely pseudo-simple, simply connected $k$-group $G$, the comparison map
$$i_G\colon G\to \mathrm{Res}_{k'/k}(G_{k'}/\mathscr{R}_{\mathrm{u},\,k'}(G_{k'}))$$
is an isomorphism when $\mathrm{char}(k)>3.$ We recall the Shapiro lemma \cite[Chapitre IV, 2.3, Lemme]{Oes84} (see also the more general result from \cite[Lemma 4.1.1]{BCS25}): for a finite field extension $k'/k$ and a smooth, affine $k'$-group $G'$,
$$\mathrm{H}^1(k,\mathrm{Res}_{k'/k}(G'))=\mathrm{H}^1(k',G'),$$
where $\mathrm{Res}_{k'/k}(G')$ denotes the restriction of scalar of $G'$ over $k'/k.$ Thus, we could reduce to the classical Serre conjecture II. When $\mathrm{char}(k)<3,$ we will need the study of basic exotic groups and basic non-reduced groups. More precisely, using the classification theory of pseudo-reductive groups developed in \cite{CGP15, CP16}, we could show that if $i_G$ is not an isomorphism then there is a finite field extension $k'/k$ and a $k'$-group $G'$ such that 
$$G\simeq \mathrm{Res}_{k'/k}(G'),$$
where $G'$ is either basic exotic or basic non-reduced. The conclusion is then followed by the triviality of torsor under basic non-reduced groups and a bijection between the set of torsors under basic exotic groups and that of semisimple, simply connected groups. 

This article is organized as follows. In section $2$, we review the classification of pseudo-reductive groups. In section $3$, we prove a decomposition result and a structure result for pseudo-semisimple, simply connected groups. As an application, we derive the main theorem \ref{main-thm-intro}.
\epp 
\bpp[Conventions and notation] For a field $k,$ let $k^s$ be a choice of its separable closure. The \emph{degree of imperfection} of $k$ when $p\coloneqq \mathrm{char}(k)>0,$ is a nonnegative integer $r$ (if it exists) such that $k/k^p$ is of degree $p^r.$ If $\mathrm{char}(k)=0,$ the degree of imperfection is $0.$ For a prime $\ell,$ the \emph{$\ell$-cohomological dimension} of $k$ is the smallest nonnegative integer $n$ such that for every torsion $\mathrm{Gal}(k^s/k)$-module $A$, the $\ell$-primary torsion subgroup of its $i$th Galois cohomology group  $\mathrm{H}^i(k^s/k,A)$ is trivial for all $i>n.$ If no such $n$ exists, the \emph{$\ell$-cohomological dimension} of $k$ is~$+\infty.$ The \emph{cohomological dimension} of $k$ is the supremum of all of its $\ell$-cohomological dimensions for different primes $\ell.$

The \emph{$k$-unipotent radical} of a connected, affine, smooth $k$-group $G$, denoted by $\mathscr{R}_{\mathrm{u},\,k}(G)$, is the maximal connected, smooth, unipotent normal $k$-subgroup of $G.$ A connected, affine, smooth $k$-group $G$ is
\begin{itemize}
    \item \emph{pseudo-reductive} if $\mathscr{R}_{\mathrm{u},\,k}(G)=1$;
    \item \emph{perfect} if it is equal to its derived subgroup;
    \item \emph{simply connected} if the reductive $\overline{k}$-group $G_{\overline{k}}/\mathscr{R}_{\mathrm{u},\,\overline{k}}(G_{\overline{k}})$ is perfect and simply connected.
\end{itemize}
A pseudo-reductive $k$-group is \emph{pseudo-semisimple} if it is perfect. A connected, affine, smooth $k$-group is \emph{pseudo-simple} if it is non-commutative and it does not contain any connected, smooth, normal subgroup except $1$ and itself. A $k$-group $G$ is \emph{absolutely pseudo-simple} if $G_{k^s}$ is pseudo-simple.
\epp 
\subsection*{Acknowledgements} I would like to thank my advisor K\k{e}stutis \v{C}esnavi\v{c}ius for his encouragement, helpful discussions and suggestions. Most of the result presented here was carried out during the author’s stay at the Sorbonn\'e Universit\'e, Institut de Math\'ematiques de Jussieu-Paris Rive Gauche. I would like to thank K\k{e}stutis \v{C}esnavi\v{c}ius again for his invitation and support during my visit.
\section{Recollections on pseudo-reductive groups}
Throughout this section, $k$ is a field and $p\coloneqq \mathrm{char}(k)$. The key result (cf.~Lemma \ref{structure-standard}) is that every absolutely pseudo-simple, simply connected $k$-group is isomorphic to the restriction of scalars of some semisimple, simply connected group defined over a finite field extension of $k,$ when $\mathrm{char}(k)>3.$ The mechanism
 to measure how a pseudo-reductive group is near to being a restriction of scalars of a reductive group is the comparision map. We will review the comparison map and the key result in subsection \ref{comp-map}.
\bpp[The comparison map $i_G$]\label{comp-map} 
Let $G$ be a connected, affine, smooth $k$-group and let $k'$ be the field of definition of  $\mathscr{R}_{\mathrm{u},\,\overline{k}}(G_{\overline{k}}).$ There is a natural map $$i_G\colon G\to \mathrm{Res}_{k'/k}(G_{k'}^{\mathrm{red}})$$ corresponding to the quotient map $G_{k'} \to G_{k'}^{\mathrm{red}}\coloneqq G_{k'}/\mathscr{R}_{\mathrm{u},\,k'}(G_{k'}).$ We call $i_G$ the \emph{comparison map} of $G.$ A pseudo-reductive $k$-group $G$ is \emph{of minimal type} if $\ker(i_G)\cap Z_G(T)=1$ for some maximal $k$-torus $T\subset G$ (equivalently, for every maximal $k$-torus $T\subset G$, cf. \cite[Proposition~9.4.2]{CGP15}). By \cite[Remark~4.3.2, Proposition~5.3.3]{CP16}, every pseudo-semisimple, simply connected $k$-group is of minimal type except possibly when $p=2$ and $[k:k^2]\ge 4.$   

Let $k'$ be a finite reduced $k$-algebra. Let $G'$ be a $k'$-group such that its fiber over each factor field of $k'$ is reductive and let $T'$ be a maximal $k'$-torus of $G'$. The conjugation action of $T'$ on $G'$ induces an action of $T'/Z_{G'}$ on $G'.$ Via the functoriality, $\mathrm{Res}_{k'/k}(T'/Z_{G'})$ naturally acts on $\mathrm{Res}_{k'/k}(G).$ Given a pair $(C,\phi)$ of a commutative, pseudo-reductive $k$-group $C$ and a $k$-homomorphism $\phi\colon  \mathrm{Res}_{k'/k}(T')\to C$ such that the $k$-homomorphism $\mathrm{Res}_{k'/k}(T')\to \mathrm{Res}_{k'/k}(T'/Z_{G'})$ factors through $\phi$, the $k$-group $C$ acts on $\mathrm{Res}_{k'/k}(G')$ through the $k$-homomorphism $C\to \mathrm{Res}_{k'/k}(T'/Z_{G'}).$ Consider the inclusion $\iota \colon \mathrm{Res}_{k'/k}(T')\hookrightarrow \mathrm{Res}_{k'/k}(G')$ and the twisted diagonal map
$$\alpha\colon \mathrm{Res}_{k'/k}(T')\to \mathrm{Res}_{k'/k}(G')\rtimes C, \quad t'\mapsto (\iota(t')^{-1},\phi(t')).$$
By \cite[Proposition 1.4.3]{CGP15}, $\mathrm{coker}(\alpha)$ is a pseudo-reductive $k$-group. Every $k$-group $G$ that arises as $\mathrm{coker}(\alpha)$ for some $(G',T',C,\phi)$ is called \emph{standard}. The next lemma states that if an absolutely pseudo-simple, simply connected $k$-group $G$ is standard and the root system of $G_{k^s}$ is reduced then $G$ is isomorphic to a restriction of scalars of an absolutely simple, simply connected $k'$-group for some purely inseparable finite extension $k'$ of $k$. We also recall here some criteria for standardness and reducedness of root system.
\epp 
\begin{lem}\label{structure-standard} Let $G$ be a pseudo-reductive $k$-group and set $\overline{G}\coloneqq G_{\overline{k}}^{\mathrm{red}}.$
\benum
    \item \label{rs-reduced} The root system of $G_{k^s}$ is reduced unless $k$ is imperfect, $p=2$ and $\overline{G}$ has a connected, simple, semisimple, simply connected normal $\overline{k}$-subgroup of type $\mathrm{C}_n$ with $n\ge 1$.
    \item \label{standard}The $k$-group $G$ is standard except possibly when $p\in \{2,3\}$ and the Dynkin diagram of $\overline{G}$ either contains an edge with multiplicity $p$ or has an isolated vertex when $p=2.$
    \item \label{comp-isom}If $G$ is an absolutely pseudo-simple, simply connected, standard and the root system of $G_{k^s}$ is reduced, then the comparison map $i_G$ is an isomorphism.
\eenum
\end{lem}
\begin{proof} The claim \ref{rs-reduced} on reducedness of the root system of $G_{k^s}$ is \cite[Theorem 2.3.10]{CGP15} and the claim \ref{standard} on standardness of $G$ is \cite[Theorem 5.1.1(1)]{CGP15}. The last claim \ref{comp-isom} is~\cite[Proposition~3.2.7]{CP16}.
\end{proof}  
By Lemma \ref{structure-standard}, when $p>3$, the comparison map of an absolutely pseudo-simple, simply connected $k$-group is an isomorphism. The next subsections briefly review the theory of absolutely pseudo-simple, simply connected groups such that the corresponding comparison maps are not isomorphisms. By Lemma \ref{structure-standard}, these groups arise when the group itself is not standard or the root system over $k^s$ is not reduced. 
\bpp[Exotic and basic exotic groups] A pseudo-reductive $k$-group that is not standard could decompose into a product of a standard $k$-group and an \emph{exotic $k$-group} provided $p=3$ (cf. Lemma \ref{decomposition-lemma-char3}). 
 
Suppose that $p \in \{2,3\}$. Let $G$ be an absolutely simple, simply connected $k$-group with Dynkin diagram having an edge with multiplicity $p$. By \cite[Lemma 7.1.2]{CGP15}, there is a simply connected, absolutely simple $k$-group $\widetilde{G}$ such that the relative Frobenius isogeny $F_{G/k}\colon G\to G^{(p)}$ is factorized uniquely into the composition of two $k$-isogenies $$G\xrightarrow{\pi}\widetilde{G}\to G^{(p)}$$
such that $\pi$ is non-central and has no nontrivial factorization. We call $\pi:G\to \widetilde{G}$ a \emph{very special $k$-isogeny} of $G.$

Assume that $k$ is imperfect and $p\in \{2,3\}$. A smooth $k$-group $G$ is \emph{basic exotic} (cf. \cite[Definition 7.2.6, Proposition 7.3.1]{CGP15}) if there exists a finite field extension $k'/k$ with $k'^p\subset k$, a very special $k'$-isogeny $\pi'\colon G'\to \widetilde{G'}$ and a Levi $k$-subgroup $\overline{G}$ of $\mathrm{Res}_{k'/k}(\widetilde{G'})$ such that $G\simeq f^{-1}(\overline{G}),$ where $f\coloneqq \mathrm{Res}_{k'/k}(\pi').$ By \cite[Lemma 7.2.1]{CGP15}, for a reductive $k'$-group $G',$ a $k$-group $\mathrm{Res}_{k'/k}(G')$ admits a Levi $k$-subgroup if and only if $G'$ is defined over $k.$ By \cite[Lemma 7.2.1, Theorem 7.2.3, Proposition 7.3.1]{CGP15}, every basic exotic $k$-group is absolutely pseudo-simple. Basic exotic $k$-groups are never standard \cite[Proposition 8.1.1]{CGP15}. 

An \emph{exotic $k$-group} is a $k$-group of the form $\mathrm{Res}_{k'/k}(G')$ for a finite reduced $k$-algebra $k'$ and a $k'$-group $G'$ whose fibers are basic exotic. There is a decomposition lemma \cite[Theorem 8.2.10]{CGP15} for pseudo-reductive groups over fields of characteristic $3$. 
\epp 
\begin{lemma}\label{decomposition-lemma-char3} For $p=3$ and every pseudo-reductive $k$-group $G$, there is a unique decomposition  $$G=G_1\times G_2$$ such that $G_1$ is standard and $G_2$ is either trivial or exotic.  
\end{lemma}
When $k'=k^{1/p}$ and $k$ is of finite degree of imperfection, the following result \cite[Proposition 7.3.3(1)]{CGP15} reduces the study of torsors under basic exotic $k$-groups to that of semisimple, simply connected $k$-groups.
\begin{lemma}\label{torsor-basic-exotic} Let $k$ be an imperfect field of finite degree of imperfection and $p\in \{2,3\}.$ Let $G$ be a basic exotic $k$-group arising from $(G',k',\overline{G})$ such that  $k'=k^{1/p}.$ Then the natural map 
$$\mathrm{H}^1(k,G)\to \mathrm{H}^1(k,\overline{G})$$
is bijective.
\end{lemma}

\bpp[Totally non-reduced and basic non-reduced groups]\label{basic-non-reduced} A pseudo-reductive $k$-group $G$ is \emph{totally non-reduced} if $G$ is perfect and every irreducible component of a root system of $G_{k^s}$ is not reduced. A \emph{basic non-reduced} $k$-group is an absolutely pseudo-simple $k$-group such that the root system of $G_{k^s}$ is not reduced and the field of definition of $\mathscr{R}_{\mathrm{u},\,\overline{k}}(G_{\overline{k}})$ is quadratic over $k.$ By \ref{structure-standard}\ref{rs-reduced}, these groups occur only when $k$ is imperfect and $p=2.$

Suppose from now that $p=2$ and $[k:k^2]=2.$ By \cite[Proposition 10.1.4]{CP16}, every totally non-reduced $k$-group $G$ is isomorphic to $\mathrm{Res}_{k'/k}(G')$ for a finite reduced $k$-algebra $k'$ and a $k'$-group $G'$ whose fibers are basic non-reduced, which is analogous to the concept of exotic and basic-exotic groups. We also have the decomposition lemma \cite[Proposition 10.1.6]{CGP15} for pseudo-reductive groups that has non-reduced root system over $k^s.$ 
\epp 
\begin{lemma}\label{decomposition-lemma-char2} For $p=2,$ $[k:k^2]=2$ and a pseudo-reductive $k$-group $G$ such that the root system of $G_{k^s}$ is not reduced, there is a unique decomposition $$G=G_1\times G_2$$ such that $G_1$ is totally non-reduced and $(G_2)_{k^s}$ has a reduced root system.
\end{lemma} 
We conclude with the vanishing lemma \cite[Proposition 9.9.4]{CGP15} of $\mathrm{H}^1(k,G)$.
\begin{lemma}\label{non-reduced-vanish}
For $p=2,$ $[k:k^2]=2$ and an absolutely pseudo-simple $k$-group $G$ such that the root system of $G_{k^s}$ is not reduced, 
$$\mathrm{H}^1(k,G)=\{\ast\}.$$
\end{lemma}

\section{Structure of pseudo-semisimple, simply connected groups}
Throughout this section, $k$ is a field and $p\coloneqq \mathrm{char}(k)$.

The common strategy to study semisimple, simply connected is to use its decomposition into simple factors. We adapt this strategy to ``pseudo-reductive'' setting. 
\begin{lemma}\label{decomp-thm} Suppose that $[k:k^2]\le 2$ when $p=2$. For every pseudo-semisimple, simply connected $k$-group $G$, there is an isomorphism of $k$-groups 
$$G\simeq \prod_{i=1}^n G_i$$ such that each $G_i$ is pseudo-simple, simply connected. Moreover, the $k$-groups $G_i$'s are unique up to permutation and isomorphism. 
\end{lemma}
\begin{proof} Every pseudo-semisimple, simply connected $k$-group $G$ is of minimal type (cf.~\ref{comp-map}). The decomposition is then followed by \cite[Proposition 5.3.4]{CP16}. 
\end{proof}
\begin{cor}\label{pseudo-simply-connected}
Suppose that $[k:k^2]\le 2$ when $p=2.$ For every pseudo-semisimple, simply connected $k$-group $G$, there is a finite \'etale $k$-algebra $k'$ and an affine, smooth $k'$-group $G'$ whose fiber over each factor field of $k'$ is absolutely pseudo-simple, simply connected such that there is an isomorphism of $k$-groups 
$$G\simeq \mathrm{Res}_{k'/k}(G').$$  
\end{cor}
This follows from a standard Galois descent argument. However, for the sake of completion, we give a complete proof here.  
\begin{proof} Applying Lemma \ref{decomp-thm} for $G_{k^s},$ there are pseudo-simple, simply connected $k^s$-groups $G_i$ such that 
$$G_{k^s}\simeq \prod_{i=1}^n G_i.$$
The Galois group $ \mathrm{Gal}(k^s/k)$ acts on the set $\{G_i:1\le i\le n\}$ by the uniqueness of the decomposition. For $1\le i\le n$, the stabilizer of $G_i$ is $\mathrm{Gal}(k^s/k_i)$ for some finite field extension of $k_i/k.$ So $G_i$ is defined over $k_i$ and the product of simple factors in its orbit is isomorphic to $$\prod_{\sigma \in \mathrm{Gal}(k^s/k)/\mathrm{Gal}(k^s/k_i)} G_i^\sigma \simeq (\mathrm{Res}_{k_i/k}(G_i))_{k_i} .$$ 
In other words, such product descends to a $k$-group $\mathrm{Res}_{k_i/k}(G_i).$ This gives the desired presentation since $G_i$ is absolutely pseudo-simple over $k_i$.
\end{proof}
After reducing to absolutely pseudo-simple, simply connected $k$-group, we prove the following structure theorem for absolutely pseudo-simple, simply connected groups. 
\begin{thm}\label{structure-abs} Suppose that if $p=2$ then $[k:k^2]\le 2$. For an absolutely pseudo-simple, simply connected $k$-group $G$, there is a finite extension $k'/k$ and a $k'$-group $G'$ such that $$G\simeq \mathrm{Res}_{k'/k}(G'),$$ 
where $G'$ is either
\begin{itemize}
    \item a semisimple, simply connected $k'$-group;
    \item a basic exotic $k'$-group (exists only when $\mathrm{char}(k)\in \{2,3\}$); 
    \item a basic non-reduced $k'$-group (exists only when $\mathrm{char}(k)=2$).
\end{itemize}
\end{thm}
\begin{proof} If $p>3,$ then this follows by Lemma \ref{structure-standard}. If $p=3,$ then by Lemma \ref{decomposition-lemma-char3}, $G$ is either standard or exotic. If $G$ is standard, then we could utilize Lemma \ref{structure-standard}\ref{comp-isom}. If $G$ is exotic, then by pseudo-simplicity, $G=\mathrm{Res}_{k'/k}(G')$ for some finite field extension $k'/k$ and basic exotic $k'$-group $G'.$ Suppose now that $p=2.$ By Lemma \ref{decomposition-lemma-char2}, either $G$ is totally non-reduced or $G_{k^s}$ has a reduced root system. In the former case, $G\simeq \mathrm{Res}_{k'/k}(G')$ for some finite extension $k'/k$ and basic non-reduced $k'$-group $G'$ (cf.~\ref{basic-non-reduced}). One may assume that $G_{k^s}$ has a reduced root system. If $G$ is standard, we could utilize Lemma \ref{structure-standard}\ref{comp-isom} again. Otherwise, by \cite[Proposition 11.1.4]{CGP15}, there is a finite extension $k'/k$ and a basic exotic $k'$-group $G'$ such that $G\simeq \mathrm{Res}_{k'/k}(G')$. 
\end{proof}
We conclude with the proof of the main result of this article.
\bcor\label{main-thm} Let $k$ be a field of cohomological dimension at most $2$ and has degree of imperfection at most $1$. The following are equivalent
\benumr 
\item  $\mathrm{H}^1(k,G)=\{\ast\}$ for every semisimple, simply connected $k$-group $G;$
    \item $\mathrm{H}^1(k,G)=\{\ast\}$ for every pseudo-semisimple, simply connected $k$-group $G.$
\eenum
\ecor 
\begin{proof} Recall that cohomological dimension of algebraic extension of $k$ is not increasing and the imperfection degree is insensitive passing to finite extensions. Notice also that if $k'/k$ is a finite field extension and $G'$ is a smooth, affine $k'$-group, then 
$$\mathrm{H}^1(k,\mathrm{Res}_{k'/k}(G'))=\mathrm{H}^1(k',G').$$
Let $G$ be a pseudo-semisimple, simply connected $k$-group. Thanks to Corollary \ref{pseudo-simply-connected}, one could assume that $G$ is absolutely pseudo-simple, simply connected. By the structure theorem \ref{structure-abs}, we are reducing to the case where $G$ is either semisimple, simply connected $k$-group or basic exotic or basic non-reduced. By Lemma \ref{torsor-basic-exotic} and Lemma \ref{non-reduced-vanish}, we reduce to the semisimple, simply connected case, as desired.
\end{proof}

\end{document}